\documentclass[11pt]{amsart}
\usepackage{eurosym,amscd,amsmath,amsfonts,amssymb,amsthm,color,graphicx,latexsym,geometry,hyperref}

\usepackage[norefs,nocites]{refcheck}

\hypersetup{
colorlinks=true,
linkcolor=blue,
citecolor=cyan
}

\providecommand{\U}[1]{\protect\rule{.1in}{.1in}}

\newtheorem{theorem}{Theorem}[section]
\newtheorem{proposition}[theorem]{Proposition}

\newtheorem{example}[theorem]{Example}

\newtheorem{remark}[theorem]{Remark}

\newtheorem{lemma}[theorem]{Lemma}

\numberwithin{equation}{section}

\begin{document}
\title{On summability of multilinear operators and applications}

\author[Albuquerque]{N. Albuquerque}
\address[N. Albuquerque]{Departamento de Matem\'{a}tica\\
\indent Universidade Federal da Para\'{\i}ba\\
\indent58.051-900 - Jo\~{a}o Pessoa, Brazil}
\email{ngalbuquerque@mat.ufpb.br and ngalbuquerque@pq.cnpq.br}

\author[Ara\'ujo]{G. Ara\'ujo}
\address[G. Ara\'ujo]{Departamento de Matem\'{a}tica\\
\indent Universidade Estadual da Para\'{\i}ba\\
\indent58.429-500 - Campina Grande, Brazil}
\email{gdasaraujo@gmail.com and gustavoaraujo@cct.uepb.edu.br}

\author[Cavalcante]{W. Cavalcante}
\address[W. Cavalcante]{Departamento de Matem\'{a}tica - Federal University of
Pernambuco - Recife - Brazil}
\email{wasthenny@dmat.ufpe.br}

\author[Nogueira]{T. Nogueira}
\address[T. Nogueira]{Departamento de Matem\'{a}tica\\
Universidade Federal da Para\'{\i}ba\\
58.051-900 - Jo\~{a}o Pessoa, Brazil \\
and Departamento de Ci\^{e}ncias Exatas Tecnol\'{o}gicas e Humanas\\
Universidade Federal Rural do Semi-\'{A}rido \\
59.515-000 - Angicos, Brazil.}
\email{tonykleverson@gmail.com and tony.nogueira@ufersa.edu.br}

\author[N\'u\~nez]{D. N\'u\~nez-Alarc\'on}
\address[D. N\'u\~nez-Alarc\'on]{Departamento de Matem\'atica\\
\indent Universidad Nacional de Colombia\\
\indent 111321 - Bogot\'a, Colombia.}
\email{danielnunezal@gmail.com}

\author[Pellegrino]{D. Pellegrino}
\address[D. Pellegrino]{Departamento de Matem\'{a}tic,a\\
\indent Universidade Federal da Para\'{\i}ba\\
\indent 58.051-900 - Jo\~{a}o Pessoa, Brazil.}
\email{dmpellegrino@gmail.com and pellegrino@pq.cnpq.br}

\author[Rueda]{P. Rueda}
\address[P. Rueda]{Departamento de An\'{a}lisis Matem\'{a}tico\\
\indent Universidad de Valencia\\
\indent 46100 Burjassot, Valencia.}
\email{pilar.rueda@uv.es}


\thanks{N. Albuquerque is supported by Capes, CNPq 409938/2016-5. T. Nogueira and W. Cavalcante are supported by Capes. D. Pellegrino is supported by CNPq; P. Rueda and D. Pellegrino are supported by Ministerio de Econom\'{\i}a, Industria y Competitividad and FEDER under project MTM2016-77054-C2-1-P. Part of this work was done while P. Rueda was visiting the Department of Mathematical Sciences at Kent State University suported by Ministerio de Educaci\'on, Cultura y Deporte PRX16/00037. She thanks this Department for its kind hospitality.}


\begin{abstract}
This paper has two clear motivations: a technical and a practical. The technical motivation unifies in a single and crystal clear formulation a huge family of inequalities that have been produced separately in the last 90 years in different contexts. But we do not just join inequalities; our method also create a family of inequalities invisible by previous approaches. The practical motivation is to show that our deeper approach has strength to attack various problems. We provide new applications of our family of inequalities, continuing the recent work by Maia et al., that, by using our main theorem, substantially improved an inequality of Carando et al. which seemed impossible to be achieved by their original method.
\end{abstract}

\maketitle

\tableofcontents

\section{Introduction}

Absolutely summing linear operators (see \cite{diestel}) can be generalized to the multilinear framework by several different approaches. There is a vast recent literature in this line (see \cite{popa, popa1, popa6} and the references therein) and also some works attempting to unify different approaches (see \cite{bc, bpr, serrano}).

The following definition is perhaps the most general approach, recently proposed in \cite{bayart}: Let $m\geq1$, $E_{1},\dots,E_{m}$, $F$ be Banach spaces and $T:E_{1}\times\cdots\times E_{m}\rightarrow F$ be an $m$-linear operator. Let also $\Lambda\subset\mathbb{N}^{m}$. For $r\in(0,\infty)$ and $p\geq1$, we say that $T$ is $\Lambda-(r,p)-$summing if there exists a constant $C>0$ such that for all sequences $x(j)\subset E_{j}^{\mathbb{N}}$, $1\leq j\leq m$,
\[
\left(  \sum_{\mathbf{i}\in\Lambda}\Vert T(x_{\mathbf{i}})\Vert^{r}\right)^{\frac{1}{r}}\leq C\left\Vert x(1)\right\Vert _{w,p}\cdots\left\Vert x(m)\right\Vert _{w,p},
\]
where $T(x_{\mathbf{i}})$ stands for $T(x_{i_{1}}(1),\dots,x_{i_{m}}(m))$ and $\left\Vert x\right\Vert _{w,p}$ stands for the weak $\ell_{p}$-norm of $x$ defined by
\[
\left\Vert x\right\Vert _{w,p}=\sup_{\Vert x^{\ast}\Vert\leq1}\left(\sum_{i=1}^{\infty}|x^{\ast}(x_{i})|^{p}\right)  ^{\frac{1}{p}}.
\]

When $\Lambda=\mathbb{N}^{m}$, we recover the notion of a $(r,p)-$multiple summing map introduced in \cite{bombal,matos}. When $\Lambda=\{(n,\dots,n):n\in\mathbb{N}\}$, we get the definition of a $(r,p)$-absolutely summing maps which was introduced in \cite{AM}. We shall denote by $\pi_{r,p}^{abs}$ this class.

The cases $\Lambda=\mathbb{N}^{m}$ and $\Lambda=\{(n,\dots,n): n\in
\mathbb{N}\}$ are very well studied in the literature (see, for instance,
\cite{popa5, popa6} just to cite some references); in this paper we investigate intermediary situations,
i.e., the cases of sets $\Lambda$ strictly located between $\{(n,\dots,n): n\in
\mathbb{N}\}$ and $\mathbb{N}^{m}$.

For $p\in\lbrack1,\infty]$, as usual, we consider the Banach spaces of weakly $p$-summable sequences
\[
\ell_{p}^{w}(E):=\left\{  (x_{j})_{j=1}^{\infty}\subset E:\left\Vert
(x_{j})_{j=1}^{\infty}\right\Vert _{w,p}<\infty\right\}
\]
and strongly $p$-summable sequences
\[
\ell_{p}(E):=\left\{  (x_{j})_{j=1}^{\infty}\subset E:\left\Vert (x_{j}%
)_{j=1}^{\infty}\right\Vert _{p}:=\left(\sum_{j=1}^\infty \|x_j\|^p\right)^{\frac 1p}<\infty\right\}  .
\]
All along this paper, the topological dual of $E$ is denoted by
$E^{\ast}$ and the conjugate of $1\leq p\leq \infty$ is represented by $p^{\ast}$,
i.e., $\frac{1}{p}+\frac{1}{p^{\ast}}=1$. As usual, $e_{j}$ are
the canonical vectors and%
\[
\left\Vert T\right\Vert :=\sup_{\left\Vert x_{1}\right\Vert ,\ldots,\left\Vert
x_{m}\right\Vert \leq1}\left\Vert T(x_{1},\ldots,x_{m})\right\Vert
\]
for any continuous $m$-linear mapping $T:E_{1}\times\cdots\times
E_{m}\rightarrow F$.
Henceforth $\mathcal{L}(E_{1},\dots,E_{m};F)$ stands for the Banach space of
all bounded $m$-linear operators from $E_{1}\times\cdots\times E_{m}$ to $F$
endowed with this $\sup$ norm.

The canonical isometric isomorphisms (see \cite[Proposition 2.2]{diestel})  $\mathcal{L}(\ell_{p^{\ast}},E)=\ell
_{p}^{w}(E)$ and $\mathcal{L}(c_{0},E)=\ell_{1}^{w}(E)$ tells us that certain
cases of summability of multilinear operators are equivalent to investigate
\[
\left(  \sum_{\mathbf{i}\in\Lambda}\Vert T(e_{\mathbf{i}})\Vert^{r}\right)
^{\frac{1}{r}}\leq C\Vert T\Vert,
\]
for $T:\ell_{p}\times\cdots\times\ell_{p}\rightarrow F$ or $T:c_{0}%
\times\cdots\times c_{0}\rightarrow F$ and this is precisely when the theory
of Hardy--Littlewood inequalities meets the theory of absolutely summing
multilinear operators.

%\subsection{Statement of the results}

Results related to summability of multilinear operators date back, at least,
to the 30's, when Littlewood proved his seminal $4/3$ inequality. Since then,
several different related results and approaches have appeared, as the
Bohnenblust--Hille (\textit{Annals of Math.}, 1931) and Hardy--Littlewood
(\textit{Quarterly J. Math.}, 1934) inequalities, that can be considered two
keystones of the theory for multilinear operators. In the last
30 years, several multilinear variants of these classical inequalities have
appeared. Let us classify them depending on whether the involved sum is done
in one or all indices.

Let $\mathbb{K}$ be $\mathbb{R}$ or $\mathbb{C},$ $m$ be a positive integer
and $1\leq p_{1},...,p_{m}\leq\infty$. From now on, for $\mathbf{p}%
:=(p_{1},\ldots,p_{m})\in\lbrack1,+\infty]^{m}$, let
\[
\left\vert \frac{1}{\mathbf{p}}\right\vert :=\frac{1}{p_{1}}+\cdots+\frac
{1}{p_{m}}.
\]
We shall also denote $X_{p}:=\ell_{p}$ for $1\leq p<\infty$, and $X_{\infty}:=c_{0}$.

\bigskip

\noindent I - Sums in one index ($\Lambda=\{(n,\dots,n):n\in\mathbb{N}\}$):

\bigskip

\begin{itemize}
\item Aron and Globevnik (\cite{agg}, 1989): For every continuous $m$-linear
form $T:c_{0}\times\cdots\times c_{0}\rightarrow\mathbb{K}$,
\begin{equation}
\sum_{i=1}^{\infty}\left\vert T(e_{i},\ldots,e_{i})\right\vert \leq\left\Vert
T\right\Vert . \label{aronglob}%
\end{equation}

\item Zalduendo (\cite{zal}, 1993): Let $\left\vert \frac{1}{\mathbf{p}%
}\right\vert <1$. For every continuous $m$-linear form $T:X_{p_{1}}%
\times\cdots\times X_{p_{m}}\rightarrow\mathbb{K}$,
\begin{equation}
\left(  \sum\limits_{i=1}^{\infty}\left\vert T\left(  e_{i},...,e_{i}\right)
\right\vert ^{\frac{1}{1-\left\vert \frac{1}{\mathbf{p}}\right\vert }}\right)
^{1-\left\vert \frac{1}{\mathbf{p}}\right\vert }\leq\Vert T\Vert. \label{8822}%
\end{equation}
\end{itemize}

\bigskip

\noindent II - Sums in all indices ($\Lambda=\mathbb{N}^{m}$):

\bigskip

\begin{itemize}
\item Bohnenblust--Hille inequality (\cite{bh}, 1931): There exists a constant $C_{m,\infty}^{\mathbb{K}}\geq1$ such that, for every
continuous $m$--linear form $T:c_{0}\times\cdots\times c_{0}\rightarrow
\mathbb{K}$,
\begin{equation}
\left(  \sum_{i_{1},\ldots,i_{m}=1}^{\infty}\left\vert T(e_{i_{1}}%
,\ldots,e_{i_{m}})\right\vert ^{\frac{2m}{m+1}}\right)  ^{\frac{m+1}{2m}}\leq
C_{m,\infty}^{\mathbb{K}}\left\Vert T\right\Vert . \label{u88}
\end{equation}

\item Hardy--Littlewood (\cite{hardy}, 1934) and Praciano-Pereira (\cite{pra}, 1981): Let $\left\vert \frac{1}{\mathbf{p}}\right\vert \leq\frac{1}{2}$. There
exists a constant $C_{m,\mathbf{p}}^{\mathbb{K}}\geq1$ such that, for every
continuous $m$-linear form $T:X_{p_{1}}\times\cdots\times X_{p_{m}}%
\rightarrow\mathbb{K}$,
\begin{equation}
\left(  \sum_{i_{1},\ldots,i_{m}=1}^{\infty}\left\vert T(e_{i_{1}}%
,\ldots,e_{i_{m}})\right\vert ^{\frac{2m}{m+1-2\left\vert \frac{1}{\mathbf{p}%
}\right\vert }}\right)  ^{\frac{m+1-2\left\vert \frac{1}{\mathbf{p}%
}\right\vert }{2m}}\leq C_{m,\mathbf{p}}^{\mathbb{K}}\left\Vert T\right\Vert .
\label{i99}
\end{equation}

\item Hardy--Littlewood (\cite{hardy}, 1934) and Dimant--Sevilla-Peris (\cite{dimant}, 2016): Let $\frac{1}{2}\leq\left\vert \frac{1}{\mathbf{p}%
}\right\vert <1$. There exists a constant $D_{m,\mathbf{p}%
}^{\mathbb{K}}\geq1$ such that
\begin{equation}
\left(  \sum_{i_{1},\dots,i_{m}=1}^{\infty}\left\vert T(e_{i_{1}}%
,\dots,e_{i_{m}})\right\vert ^{\frac{1}{1-\left\vert \frac{1}{\mathbf{p}%
}\right\vert }}\right)  ^{1-\left\vert \frac{1}{\mathbf{p}}\right\vert }\leq
D_{m,\mathbf{p}}^{\mathbb{K}}\Vert T\Vert\label{i99jagsytsb}%
\end{equation}
for every continuous $m$-linear form $T:X_{p_{1}}\times\cdots\times X_{p_{m}%
}\rightarrow\mathbb{K}$.
\end{itemize}

\bigskip

All exponents involved in the previous inequalities are sharp. An extended version of the Hardy--Littlewood/Praciano-Pereira inequality was presented in \cite{abpsisrael}:

\bigskip

\begin{itemize}
\item Albuquerque et al. (\cite{abpsisrael}, 2016):  Let $\left\vert
\frac{1}{\mathbf{p}}\right\vert \leq\frac{1}{2}$ and $\mathbf{q}:=(q_{1}%
,\dots,q_{m})\in\left[  \left(  1-\left\vert \frac{1}{\mathbf{p}}\right\vert
\right)  ^{-1},2\right]  ^{m}$. There is a constant $C_{m,\mathbf{p}%
,\mathbf{q}}^{\mathbb{K}}\geq1$ such that
\begin{equation}
\left(  \sum\limits_{i_{1}=1}^{\infty}\left(  \cdots\left(  \sum
\limits_{i_{m}=1}^{\infty}\left\vert A\left(  e_{i_{1}},\ldots,e_{i_{m}%
}\right)  \right\vert ^{q_{m}}\right)  ^{\frac{q_{m-1}}{q_{m}}}\cdots\right)
^{\frac{q_{1}}{q_{2}}}\right)  ^{\frac{1}{q_{1}}}\leq C_{m,\mathbf{p}%
,\mathbf{q}}^{\mathbb{K}}\left\Vert A\right\Vert \label{ttr}%
\end{equation}
for every continuous $m$-linear form $A:X_{p_{1}}\times\cdots\times X_{p_{m}%
}\rightarrow\mathbb{K}$ if, and only if,
\[
\frac{1}{q_{1}}+\cdots+\frac{1}{q_{m}}\leq\frac{m+1}{2}-\left\vert {\frac
{1}{\mathbf{p}}}\right\vert .
\]
\end{itemize}

\begin{remark}
Throughout all the paper, the optimal constants of each of the above inequalities will be denoted exactly as they were previously stated.
\end{remark}

\bigskip

Note that:

\begin{itemize}
\item[(a)] Zalduendo's theorem, for $p_{1}=\cdots=p_{m}=\infty$, recovers
Aron--Globevnik's theorem;

\item[(b)] The Hardy--Littlewood/Praciano-Pereira inequality, when
$p_{1}=\cdots=p_{m}=\infty$, recovers the Bohnenblust--Hille inequality;

\item[(c)] If $q_{1}=\cdots=q_{m}=\frac{2m}{m+1-2\left\vert \frac
{1}{\mathbf{p}}\right\vert }$ in \eqref{ttr}, we recover the
Hardy--Littlewood/ Praciano-Pereira inequality and we will denote
$C_{m,\mathbf{p},\left(  \frac{2m}{m+1-2\left\vert 1/\mathbf{p}\right\vert
},\ldots,\frac{2m}{m+1-2\left\vert 1/\mathbf{p}\right\vert }\right)
}^{\mathbb{K}}$ by $C_{m,\mathbf{p}}^{\mathbb{K}}$. Moreover, if $p_{1}%
=\cdots=p_{m}=p$ we will denote $C_{m,\mathbf{p}}^{\mathbb{K}}$ by
$C_{m,p}^{\mathbb{K}}$.
\end{itemize}

\bigskip

The first main objective of this article is to combine in a single formulation all the above  inequalities that were produced separately and in different contexts and that apparently did not match. We do not do this only for the mathematical beauty of unifying theories that were treated in completely different ways, but because this also provides subtle bits of information that were not accessible, such as, for example, giving a definitive answer to a problem initially considered by D. Carando et al. \cite{Carando} (this substantial improvement was recently made by Maia et al. \cite{Maia} using our main theorem). This and some other findings were only possible at the time when the theories were no longer seen separately. Despite their importance in several fields of mathematics (Quantum Information Theory, Dirichlet series, etc), the optimal constants of the $m$-linear inequalities of Bohnenblust--Hille and Hardy--Littlewood are still unknown. For the real case of the Bohnenblust-Hille  inequality it is known that the optimal constants are not contractive. As an application of our unified approach, we can analyze under what conditions we can improve the constants of such inequalities so that their constants are contractive. In fact, in Section \ref{123}, we will study how the consideration of our unified inequalities improves the Bohnenblust-Hille and Hardy-Littlewood constants so that the constants of slight variants of these inequalities become even contractive.

Let $n$ be a positive integer and from now on $e_{i}^{n}$ denotes the $n-$tuple $\left(e_{i},{...},e_{i}\right)  $. Furthermore, if $n_{1},\ldots,n_{k}\geq1$ are such that $n_{1}+\cdots+n_{k}%
=m$, then $\left(  e_{i_{1}}^{n_{1}},\ldots,e_{i_{k}}^{n_{k}}\right)  $ represents the $m-$tuple:
\[
(e_{i_{1}},\overset{\text{{\tiny $n_{1}$\thinspace times}}}{\ldots},e_{i_{1}%
},\ldots,e_{i_{k}},\overset{\text{{\tiny $n_{k}$\thinspace times}}}{\ldots
},e_{i_{k}}).
\]
The main result of this paper (Theorem \ref{hl_var}) extends and unifies
\eqref{aronglob}, \eqref{8822}, \eqref{u88}, \eqref{i99}, \eqref{i99jagsytsb}
and (\ref{ttr}), by considering intermediary setups for $\Lambda$. Theorem
\ref{hl_var} provides the following particular case whenever $p_{1}%
=\cdots=p_{m}=p$, which has a more friendly statement.

\begin{theorem}
\label{thmain2} Let $m\geq k\geq1,\,m<p\leq\infty$ and let $n_{1},\dots
,n_{k}\geq1$ be such that $n_{1}+\cdots+n_{k}=m$. Then, for every continuous
$m$--linear form $T:X_{p}\times\cdots\times X_{p}\rightarrow\mathbb{K}$, there
is a constant $M_{k,m,p}^\mathbb{K}\geq1$ such that
\[
\left(  \sum_{i_{1},\dots,i_{k}=1}^{\infty}\left\vert T\left(  e_{i_{1}%
}^{n_{1}},\dots,e_{i_{k}}^{n_{k}}\right)  \right\vert ^{\rho}\right)
^{\frac{1}{\rho}}\leq M_{k,m,p}^\mathbb{K}\left\Vert T\right\Vert ,
\]
with
\[
\rho=\frac{p}{p-m}\text{ for }m<p\leq2m\text{ and }M_{k,m,p}^\mathbb{K}\leq
D_{k,(\frac{p}{n_{1}},...,\frac{p}{n_{k}})}^{\mathbb{K}}%
\]
and
\begin{equation}
\rho=\frac{2kp}{kp+p-2m}\text{ for }p\geq2m\text{ and }M_{k,m,p}^\mathbb{K}\leq
C_{k,(\frac{p}{n_{1}},...,\frac{p}{n_{k}})}^{\mathbb{K}}. \label{765f}%
\end{equation}
Above, $C_{k,(\frac{p}{n_{1}},...,\frac{p}{n_{k}})}^{\mathbb{K}}$ and $D_{k,(\frac{p}{n_{1}},...,\frac{p}{n_{k}})}^{\mathbb{K}}$ are the constants
from \eqref{i99} and \eqref{i99jagsytsb}, respectively. Moreover, in both
cases, the exponent $\rho$ is optimal.
\end{theorem}

\begin{remark}
It seems to be interesting to stress that the optimal exponent for
the case $p>2m$ is not the exponent of the $k$-linear case. It is a kind of
combination of the cases of $k$-linear and $m$-linear forms, as it can be seen
in (\ref{765f}). In general we have the following:

\begin{itemize}
\item If $m<p<2m$ the optimal exponent depends only on $m$;
\item If $p=2m$, the optimal exponent does not depend on $m$ or $k$.
\item If $2m<p<\infty$, the optimal exponent depends on $m$ and $k$;
\item If $p=\infty$, the optimal exponent depends only on $k$.
\end{itemize}
\end{remark}

The proof of the main result combines \ two different tools based on tensor
products. Firstly, we prove a $k$-linearization method for $n$-linear
operators ($n\geq k$) which is an inductive refinement of the well known
linearization method. Secondly, we use the description of the diagonal of the
tensor product of $\ell_{p}$ spaces based on \cite[Theorem 1.3]{ArFa} and
\cite[Example 2.23(b)]{Ry}. It worths mentioning that the Zalduendo and
Aron-Globevnik inequalities can be proved in a straightforward way by means of
this technique (see Remark \ref{555}).

The search of optimal constants for the Bohnenblust-Hille inequality is an active research area nowadays (see for instance \cite{abpsisrael, apjfa,bohr,Ann} and the references therein). Very recently, our main Theorem (Theorem \ref{hl_var}) was applied in \cite{Maia} to show that the asymptotic constants of the Bohnenblust--Hille inequality for complex $m$-homogeneous polynomials whose monomials have a uniformly bounded number of variables do not depend on $m$. This is a striking result since the prior work \cite{Carando}, using a completely different technique, just  obtained constants growing polynomially with $n$. Section \ref{123} provides applications of our main result (Theorem \ref{hl_var}), in the analysis of the contractivity of the constants appearing in the inequalities when considering special sets $\Lambda$. We will prove that the Bohnenblust--Hille and Hardy--Littlewood inequalities are somewhat \textquotedblleft almost\textquotedblright\ contractive. More precisely, if $m,k,n_{1},\ldots,n_{k}\geq1$ are positive integers such that $n_{1}+\cdots+n_{k}=m$, by considering sums over the index set $\Lambda \subset \mathbb{N}^m$ that gathers all $m$-tuples
\[
\left(i_{1},
\overset{ \text{{\tiny $n_{1}$\thinspace times}} } {\ldots},
i_{1},\ldots,i_{k},
\overset{\text{{\tiny $n_{k}$\thinspace times}}} {\ldots},
i_{k}\right),
\quad i_1,\ldots,i_k \in \mathbb{N},
\]
(notice that $\Lambda$ is composed by $k$ \textquotedblleft blocks\textquotedblright) and if $k=k(m)$ is such that
\[
\lim_{m\rightarrow\infty}\frac{k\log k}{m}=0,
\]
then Theorem \ref{BH_contrac} will provide the contractivity of the Bohnenblust--Hille inequality:
\[
\lim_{m\rightarrow\infty}M_{k,m,\infty}^\mathbb{K}=1.
\]
A similar result is proved for the Hardy--Littlewood inequality (Theorem \ref{HL_contrac}).

\section{Bohnenblust--Hille and Hardy--Littlewood for block-type sets $\Lambda$}

Besides motivating the introduction of a new approach to the theory of
summability of multilinear operators, the main purpose of this section is to
present a unified version of the Bohnenblust--Hille and the Hardy--Littlewood
inequalities with partial sums (i.e., we shall consider sums allowed to run
over a set $\Lambda$ with less indices) which also recovers Zalduendo's and
Aron--Globevnik's inequalities. A tensorial perspective will present an
important role on this matter, establishing an intrinsic relationship between
the exponents and constants involved and the number of indices taken on the sums.

We shall need to introduce some other terminologies. The product
$\widehat{\otimes}_{j\in\{1,\ldots,n\}}^{\pi}E_{j}=E_{1}\widehat{\otimes}^{\pi}%
\cdots\widehat{\otimes}^{\pi}E_{n}$ denotes the completed projective $n$-fold
tensor product of $E_{1},\ldots,E_{n}$. The tensor $x_{1}\otimes\cdots\otimes
x_{n}$ is denoted for short by $\otimes_{j\in\{1,\ldots,n\}}x_{j}$, whereas
$\otimes_{n}x$ denotes the tensor $x\otimes\cdots\otimes x$. In a similar way,
$\times_{j\in\{1,\ldots,n\}}E_{j}$ denotes the product space $E_{1}%
\times\cdots\times E_{n}$.

Recall that $X_{p}=\ell_{p}$ if $1\leq p<\infty$ and $X_{p}=c_{0}$ if
$p=\infty$. Let $n$ be a positive integer and $1\leq p_{1},\ldots,p_{n}\leq\infty$ be such that $\frac
{1}{p_{1}}+\cdots+\frac{1}{p_{n}}<1$. From now on in this section $r,s$ are
defined by $\frac{1}{r}=\frac{1}{p_{1}}+\cdots+\frac{1}{p_{n}}$ and $\frac
{1}{s}+\frac{1}{r}=1$. Let $D_{r}\subset X_{p_{1}}\widehat{\otimes}^{\pi}%
\cdots\widehat{\otimes}^{\pi}X_{p_{n}}$ be the linear span of the tensors
$\otimes_{n}e_{i}$ and $\overline{D}_{r}$ be its closure.

Additionally, we will use the following notation: for Banach spaces $E_{1},\ldots,E_{m}$ and
an element $x\in E_{j}$, for some $j\in\{1,\ldots,m\}$, the symbol $x_{j}\cdot
e_{j}$ represents the vector $x_{j}\cdot e_{j}\in E_{1}\times\cdots\times
E_{m}$ such that its $j$-th coordinate is $x_{j}\in E_{j}$, and $0$ otherwise.

The following lemma, although known for $1\leq p_{1},\ldots,p_{n}<\infty$ (see
\cite[Theorem 1.3]{ArFa}), is the key of Theorem \ref{hl_var} and so, we give
a constructive proof inspired in \cite[Example 2.23(b)]{Ry}.

\begin{lemma}
The map $u_{r}:X_{r} \to\overline{D_{r}}$, given by $u_{r}(\sum_{i=1}^{\infty
}a_{i}e_{i})=\sum_{i=1}^{\infty}a_{i} \, \otimes_{n}\!\!e_{i}$, is an
isometric isomorphism onto.
\end{lemma}

\begin{proof}
For the sake of simplicity we will show only the case $1\leq p_{1}%
,\ldots,p_{n}<\infty$. In all the other cases, that is, when one or more
$X_{i}$'s are $c_{0}$, the proof can be easily adapted.

Let $\theta=\sum_{i=1}^{k}a_{i}\otimes_{n}\!\!e_{i}$. Using the orthogonality
of the Rademacher system, we get
\[
\theta=\int_{[0,1]^{n-1}}\otimes_{j=1}^{n-1}\left(  \sum_{i=1}^{k}%
|a_{i}|^{\frac{r}{p_{j}}}r_{i}(t_{j})e_{i}\right)  \otimes\left(  \sum
_{i=1}^{k}\mathrm{sgn}(a_{i})|a_{i}|^{\frac{r}{p_{n}}}r_{i}(t_{1})\cdots
r_{i}(t_{n-1})e_{i}\right)  dt,
\]
where $dt=dt_{1}\ldots dt_{n-1}$ and $r_{i}$ are the Rademacher functions and $\mathrm{sgn}(a)$ is the scalar
of modulus $1$ such that $\mathrm{sgn}(a)a=\left\vert a\right\vert $. Hence,
\begin{align*}
\pi\left(  \theta\right) & \leq\sup_{\underset{1\leq j\leq n-1}{0\leq t_{j}\leq1}}\left[  \prod_{j=1}^{n-1}\left\Vert \sum_{i=1}^{k}|a_{i}|^{\frac{r}{p_{j}}}r_{i}(t_{j})e_{i}\right\Vert _{p_{j}}\right]  \left\Vert \sum_{i=1}^{k}r_{i}(t_{1})\cdots r_{i}(t_{n-1})\mathrm{sgn}(a_{i})|a_{i}|^{\frac{r}{p_{n}}}e_{i}\right\Vert _{p_{n}} \\
& =\Vert(a_{i})_{i=1}^{k}\Vert_{r}.
\end{align*}

To prove $\Vert(a_{i})_{i=1}^{k}\Vert_{r}\leq\pi(\theta)$, consider the
$n$--linear form on $\ell_{p_{1}}\times\cdots\times\ell_{p_{n}}$ given by
\[
B(x^{(1)},\ldots,x^{(n)}):=\sum_{i=1}^{k}b_{i}x_{i}^{(1)}\cdots x_{i}^{(n)}%
\]
where $b_{i}={\mathrm{sgn}}(a_{i})\frac{|a_{i}|^{\frac{r}{s}}}{\Vert
(a_{i})_{i=1}^{k}\Vert_{r}^{\frac{r}{s}}}$. By H\"{o}lder's inequality,
\[
\Vert B\Vert=\sup_{\underset{1\leq j\leq n}{x^{(j)}\in B_{\ell_{p_{j}}}}%
}\left\vert \sum_{i=1}^{k}b_{i}x_{i}^{(1)}\cdots x_{i}^{(n)}\right\vert
\leq\sup_{\underset{1\leq j\leq n}{x^{(j)}\in B_{\ell_{p_{j}}}}}\Vert
(b_{i})_{i=1}^{k}\Vert_{s}\Vert x^{(1)}\Vert_{p_{1}}\cdots\Vert x^{(n)}%
\Vert_{p_{n}}=1.
\]
Therefore,
\[
\pi(\theta)\geq|\langle\theta,B\rangle|=\left\vert \sum_{i=1}^{k}a_{i}%
B(e_{i},\ldots,e_{i})\right\vert =\left\vert \sum_{i=1}^{k}a_{i}%
b_{i}\right\vert =\left(  \sum_{i=1}^{k}|a_{i}|^{r}\right)  ^{\frac{1}{r}}%
\]
and thus $\pi(\theta)=\Vert(a_{i})_{i=1}^{k}\Vert_{r}$. By extending the
isometric isomorphism to the completions, we get that $\overline{D}_{r}$ is
isometrically isomorphic to $\ell_{r}$.
\end{proof}

Using the isometry between $\overline{D_{r}}$ and $\ell_{r}$ provided in the
preceding lemma, we get:

\begin{lemma}
\label{norma1p} The sequence $\left(  \otimes_{n} e_{i}\right)  _{i\in
\mathbb{N}}$ belongs to $\ell_{s}^{w} \left(  X_{p_{1}} \widehat{\otimes}^{\pi}\cdots\widehat{\otimes}^{\pi}X_{p_{n}} \right)  $ and $$\left\Vert\left(  \otimes_{n} e_{i} \right)  _{i\in\mathbb{N}} \right\Vert _{w,s}=1.$$
\end{lemma}

\begin{proof}
Observe that
\begin{align*}
\left\Vert \left(  \otimes_{n} e_{i} \right)  _{i\in\mathbb{N}} \right\Vert
_{w,s}  &  = \sup_{\varphi\in B_{ \left(  X_{p_{1}} \widehat{\otimes}^{\pi
}\cdots\widehat{\otimes}^{\pi}X_{p_{n}}\right)  ^{\ast}}} \left(  \sum
_{i=1}^{\infty} |\varphi(\otimes_{n} e_{i})|^{s}\right)  ^{\frac{1}{s}} \\
& =
\sup_{\varphi\in B_{ \left(  \overline{D_{r}} \right)  ^{\ast}}} \left(
\sum_{i=1}^{\infty} |\varphi(\otimes_{n} e_{i})|^{s}\right)  ^{\frac{1}{s}} = \sup_{\varphi\in B_{ \ell_{s} }} \left(  \sum_{i=1}^{\infty}
|\varphi(e_{i})|^{s}\right)  ^{\frac{1}{s}} = 1.
\end{align*}
\end{proof}

The following result is a kind of $k$-\textquotedblleft
linearization\textquotedblright\ of a given $m$-linear operator and will be
used in the proof of our main result.

\begin{proposition}
\label{lin} Let $m$ be a positive integer and let $E_{1},\dots,E_{m},F$ be
Banach spaces. Let $1\leq k\leq m$ and $I_{1},\ldots,I_{k}$ be pairwise
disjoint non-void subsets of $\{1,\ldots,m\}$ such that $\cup_{j=1}^{k}%
I_{j}=\{1,\ldots,m\}$. Then given $T\in{\mathcal{L}}(E_{1},\ldots,E_{m};F)$,
there is a unique $\widehat{T}\in{\mathcal{L}}(\widehat{\otimes}_{j\in I_{1}}%
^{\pi}E_{j},\ldots,\widehat{\otimes}_{j\in I_{k}}^{\pi}E_{j};F)$ such that
\[
\widehat{T}(\otimes_{j\in I_{1}}x_{j},\dots,\otimes_{j\in I_{k}}x_{j}%
)=T(x_{1},\dots,x_{m})
\]
and $\Vert\widehat{T}\Vert=\Vert T\Vert$. The correspondence $T\leftrightarrow
\widehat{T}$ determines an isometric isomorphism between the spaces
${\mathcal{L}}(E_{1},\ldots,E_{m};F)$ and ${\mathcal{L}}(\widehat{\otimes}_{j\in
I_{1}}^{\pi}E_{j},\ldots,\widehat{\otimes}_{j\in I_{k}}^{\pi}E_{j};F)$.
\end{proposition}

\begin{proof}
We will proceed by transfinite induction on $m$. Note that for $m=1$ or $m=2$
there is nothing to be proved ($\widehat{T}$ is just the linearization of $T$
whenever $m=2$ and $k=1$). Assume that the result is true for any positive
integer less than $m$ and let $T\in{\mathcal{L}}(E_{1},\ldots, E_{m};F)$ and
$I_{1},\ldots, I_{k}$ as in the statement. Assume that $|I_{k}|=m_{k}$ and fix
$x_{j}\in E_{j}$, for any $j\in I_{k}$. Fix $\sum_{j\in I_{k}} x_{j}\cdot
e_{j}\in\times_{j\in I_{k}}E_{j}$. Consider the continuous $(m-m_{k})$-linear
mapping given by
\[
T_{\left(  \sum_{j\in I_{k}} x_{j}\cdot e_{j}\right)  } \left(  \sum_{i\in
I_{1}}x_{i}\cdot e_{i}+\cdots+\sum_{i\in I_{k-1}}x_{i}\cdot e_{i}\right)  :=
T(x_{1},\dots,x_{m}).
\]
By the induction hypothesis, there exists a unique
\[
\widetilde{T}\left(  \sum_{j\in I_{k}} x_{j}\cdot e_{j}\right)  \in
{\mathcal{L}}(\hat\otimes_{j\in I_{1}}^{\pi}E_{j}, \ldots, \hat\otimes_{j\in
I_{k-1}}^{\pi}E_{j};F)
\]
such that
\begin{align*}
& \widetilde T \left(  \sum_{j\in I_{k}} x_{j}\cdot e_{j}\right)  \left( \otimes_{i\in I_{1}}x_{i},\ldots,\otimes_{i\in I_{k-1}}x_{i}\right) \\
& = T_{\left(  \sum_{j\in I_{k}} x_{j}\cdot e_{j}\right)  } \left(  \sum_{i\in I_{1}}x_{i}\cdot e_{i}+\cdots+\sum_{i\in I_{k-1}}x_{i}\cdot e_{i} \right) = T(x_{1},\dots,x_{m})
\end{align*}
and
\[
\left\|  \widetilde T \left(  \sum_{j\in I_{k}} x_{j}\cdot e_{j} \right)
\right\|  = \left\|  T_{\left(  \sum_{j\in I_{k}} x_{j}\cdot e_{j} \right)  }
\right\|  .
\]

Define now the $m_{k}$-linear mapping $A:\times_{j\in I_{k}}E_{j}%
\rightarrow{\mathcal{L}}\left(  \widehat{\otimes}_{j\in I_{1}}^{\pi}E_{j}%
,\ldots,\widehat{\otimes}_{j\in I_{k-1}}^{\pi}E_{j};F\right)  $ given by
\[
A\left(  \sum_{i\in I_{k}}y_{i}\cdot e_{i}\right)  :=\widetilde{T}\left(
\sum_{i\in I_{k}}y_{i}\cdot e_{i}\right)
\]
and let $A_{L}\in{\mathcal{L}}\left(  \widehat{\otimes}_{j\in I_{k}}^{\pi}%
E_{j};{\mathcal{L}}\left(  \widehat{\otimes}_{j\in I_{1}}^{\pi}E_{j},\ldots
,\widehat{\otimes}_{j\in I_{k-1}}^{\pi}E_{j};F\right)  \right)  $ its
linearization, i.e., the unique linear map from $\widehat{\otimes}_{j\in I_{k}%
}^{\pi}E_{j}$ into ${\mathcal{L}}(\widehat{\otimes}_{j\in I_{1}}^{\pi}E_{j}%
,\ldots,\widehat{\otimes}_{j\in I_{k-1}}^{\pi}E_{j};F)$ such that $A_{L}\left(
\otimes_{j\in I_{k}}y_{j}\right)  =A\left(  \sum_{j\in I_{k}}y_{j}\cdot
e_{j}\right)  $. Finally, $\widehat{T}:\widehat{\otimes}_{j\in I_{1}}^{\pi}%
E_{j}\times\dots\times\widehat{\otimes}_{j\in I_{k}}^{\pi}E_{j}\rightarrow F$
defined by
\[
\widehat{T}(\theta_{1},\ldots,\theta_{k}):=A_{L}(\theta_{k})(\theta_{1}%
,\ldots,\theta_{k-1})
\]
is $k$-linear, continuous, and satisfies
\begin{align*}
\widehat{T}(\otimes_{j\in I_{1}}x_{j},\ldots,\otimes_{j\in I_{k}}x_{j})  &
=A_{L}\left(  \otimes_{j\in I_{k}}x_{j}\right)  \left(  \otimes_{j\in I_{1}%
}x_{j},\ldots,\otimes_{j\in I_{k-1}}x_{j}\right) \\
&  =\widetilde{T}\left(  \sum_{i\in I_{k}}x_{i}\cdot e_{i}\right)  \left(
\otimes_{j\in I_{1}}x_{j},\ldots,\otimes_{j\in I_{k-1}}x_{j}\right) \\
&  =T(x_{1},\ldots,x_{m})
\end{align*}
and
\begin{align*}
\Vert\widehat{T}\Vert &  =\sup_{\overset{\theta_{j}\in B_{\widehat{\otimes}_{i\in
I_{j}}^{\pi}E_{i}}}{j=1,\ldots,k}}\Vert A_{L}(\theta_{k})(\theta_{1}%
,\ldots,\theta_{k-1})\Vert\\
&  =\Vert A_{L}\Vert=\Vert A\Vert=\sup_{\overset{y_{i}\in E_{i}}{i\in I_{k}}%
}\left\Vert \tilde{T}\left(  \sum_{i\in I_{k}}y_{i}\cdot e_{i}\right)
\right\Vert =\sup_{\overset{y_{i}\in E_{k}}{i\in I_{k}}}\left\Vert
T_{(\sum_{i\in I_{k}}y_{i}\cdot e_{i})}\right\Vert =\Vert T\Vert.
\end{align*}
\end{proof}

Now we prove our main result, which unifies \eqref{aronglob}, \eqref{8822},
\eqref{u88}, \eqref{i99}, \eqref{i99jagsytsb} and (\ref{ttr}).

\begin{theorem}
\label{hl_var} Let $1\leq k\leq m$ and $n_{1},\dots,n_{k}\geq1$ be positive
integers such that $n_{1}+\cdots+n_{k}=m$ and assume that
\[
\mathbf{p}:=\left(  p_{1}^{(1)},\overset{\text{{\tiny $n_{1}$\thinspace
times}}}{\dots},p_{n_{1}}^{(1)},\ldots,p_{1}^{(k)},\overset
{\text{{\tiny $n_{k}$\thinspace times}}}{\dots},p_{n_{k}}^{(k)}\right)
\in\lbrack1,\infty]^{m}%
\]
is such that $0\leq\left\vert \frac{1}{\mathbf{p}}\right\vert <1$. Let $\mathbf{r}:=(r_1,\ldots,r_k)$ with $r_{i}$ given by $\frac{1}{r_{i}}=\frac{1}{p_{1}^{(i)}}+\cdots+\frac{1}{p_{n_{i}%
}^{(i)}}$, $i=1,\ldots,k$.

\begin{itemize}
\item[(1)] If $0\leq\left\vert \frac{1}{\mathbf{p}}\right\vert \leq\frac{1}%
{2}$ and $\mathbf{q}:=\left(  q_{1},\dots,q_{k}\right)  \in\left[  \left(
1-\left\vert \frac{1}{\mathbf{p}}\right\vert \right)  ^{-1},2\right]  ^{k}$
then, for every continuous $m$--linear form $T:\left(  \times_{1\leq i\leq
n_{1}}X_{p_{i}^{(1)}}\right)  \times\dots\times\left(  \times_{1\leq i\leq
n_{k}}X_{p_{i}^{(k)}}\right)  \rightarrow\mathbb{K}$,
\begin{equation}
\left(  \sum_{i_{1}=1}^{\infty}\left(  \cdots\left(  \sum_{i_{k}=1}^{\infty} \left\vert T\left(  e_{i_{1}}^{n_{1}},\dots,e_{i_{k}}^{n_{k}}\right)
\right\vert ^{q_{k}}\right)  ^{\frac{q_{k-1}}{q_{k}}}\dots\right)
^{\frac{q_{1}}{q_{2}}}\right)  ^{\frac{1}{q_{1}}}\leq C_{k,\mathbf{r},\mathbf{q}}^{\mathbb{K}}\left\Vert T\right\Vert
\label{clas}
\end{equation}
if and only if $\left\vert \frac{1}{\mathbf{q}}\right\vert \leq\frac{k+1}%
{2}-\left\vert \frac{1}{\mathbf{p}}\right\vert .$ In other words, the
exponents are optimal.

\item[(2)] If $\frac{1}{2}\leq\left\vert \frac{1}{\mathbf{p}}\right\vert <1$
then, for every continuous $m$--linear form $T:\left(  \times_{1\leq i\leq
n_{1}}X_{p_{i}^{(1)}}\right)  \times\dots\times\left(  \times_{1\leq i\leq
n_{k}}X_{p_{i}^{(k)}}\right)  \rightarrow\mathbb{K}$,
\begin{equation}
\left(  \sum_{i_{1},\dots,i_{k}=1}^{\infty}\left\vert T\left(  e_{i_{1}%
}^{n_{1}},\dots,e_{i_{k}}^{n_{k}}\right)  \right\vert ^{\frac{1}{1-\left\vert
\frac{1}{\mathbf{p}}\right\vert }}\right)  ^{1-\left\vert \frac{1}{\mathbf{p}%
}\right\vert }\leq D_{k,\mathbf{r}}^{\mathbb{K}%
}\left\Vert T\right\Vert . \label{hgv}%
\end{equation}
Moreover, the exponent in (\ref{hgv}) is optimal.
\end{itemize}
\end{theorem}

\begin{proof}
(1) Assume that $\left\vert \frac{1}{\mathbf{q}}\right\vert \leq\frac{k+1}%
{2}-\left\vert \frac{1}{\mathbf{p}}\right\vert .$ We shall use the notation
$$(p_{1}^{(1)},\ldots,p_{n_{1}}^{(1)},\ldots,p_{1}^{(k)},\ldots,p_{n_{k}}^{(k)})=(p_{1},\ldots,p_{m}).$$ We take the $k$-linear mapping given in
Proposition \ref{lin} $\widehat{T}:\widehat{\otimes}_{1\leq i\leq n_{1}}^{\pi
}X_{p_{i}^{(1)}}\times\dots\times\widehat{\otimes}_{1\leq i\leq n_{k}}^{\pi
}X_{p_{i}^{(k)}}\rightarrow\mathbb{K}$ , that satisfies
\[
\widehat{T}\left(  \otimes_{1\leq i\leq n_{1}}x_{i}^{(1)},\dots,\otimes_{1\leq
i\leq n_{k}}x_{i}^{(k)}\right)  =T\left(  x_{1}^{(1)},\dots,x_{n_{1}}%
^{(1)},\ldots,x_{1}^{(k)},\ldots,x_{n_{k}}^{(k)}\right)  .
\]
Then,
\[
\widehat{T}\left(  \otimes_{n_{1}}e_{i_{1}},\dots,\otimes_{n_{k}}e_{i_{k}%
}\right)  =T\left(  e_{i_{1}}^{n_{1}},\dots,e_{i_{k}}^{n_{k}}\right)  ,
\]
and $\Vert\widehat{T}\Vert=\Vert T\Vert$. Thus
\[%
\begin{array}
[c]{l}%
\displaystyle\left(  \sum_{i_{1}=1}^{\infty}\left(  \cdots\left(  \sum
_{i_{k}=1}^{\infty}\left\vert T\left(  e_{i_{1}}^{n_{1}},\dots,e_{i_{k}%
}^{n_{k}}\right)  \right\vert ^{q_{k}}\right)  ^{\frac{q_{k-1}}{q_{k}}}%
\dots\right)  ^{\frac{q_{1}}{q_{2}}}\right)  ^{\frac{1}{q_{1}}}\vspace
{0.2cm}\\
\displaystyle=\left(  \sum_{i_{1}=1}^{\infty}\left(  \cdots\left(  \sum
_{i_{k}=1}^{\infty}\left\vert \widehat{T}\left(  \otimes_{n_{1}}e_{i_{1}%
},\dots,\otimes_{n_{k}}e_{i_{k}}\right)  \right\vert ^{q_{k}}\right)
^{\frac{q_{k-1}}{q_{k}}}\dots\right)  ^{\frac{q_{1}}{q_{2}}}\right)
^{\frac{1}{q_{1}}}.
\end{array}
\]
For each $j=1,\ldots,k$, we take $u_{j}:X_{r_{j}}\rightarrow\overline
{D_{r_{j}}}$ defined by $u_{j}\left(  \sum_{i=1}^{\infty}a_{i}e_{i}\right)
=\sum_{i=1}^{\infty}a_{i}\otimes_{n_{j}}\!\!e_{i}$. Lemma \ref{norma1p} will
give
\[
\left\Vert u_{j}\right\Vert =\left\Vert \left(  \otimes_{n}e_{i}\right)
_{i\in\mathbb{N}}\right\Vert _{w,r_{j}^{\ast}}=1.
\]
Finally, it is sufficient to deal with the $k$-linear operator $S:X_{r_{1}%
}\times\cdots\times X_{r_{k}}\rightarrow\mathbb{K}$ defined by
\[
S(z_{1},\dots,z_{k}):=\widehat{T}\left(  u_{1}(z_{1}),\dots,u_{k}%
(z_{k})\right)  ,
\]
which is bounded and fulfills $\Vert S\Vert\leq\Vert\widehat{T}\Vert$.
Combining this with (\ref{ttr}) and observing that%
\[
\frac{1}{r_{1}}+\cdots+\frac{1}{r_{k}}=\left\vert \frac{1}{\mathbf{p}%
}\right\vert ,
\]
the result follows. To show that the inequalities \eqref{clas} forces the
exponent to be $\left\vert \frac{1}{\mathbf{q}}\right\vert \leq\frac{k+1}%
{2}-\left\vert \frac{1}{\mathbf{p}}\right\vert ,$ it suffices to prove by
(1.6) that
\[
\left(  \sum\limits_{j_{1}=1}^{\infty}\left(  \cdots\left(  \sum
\limits_{j_{k}=1}^{\infty}\left\vert A\left(  e_{j_{1}},\ldots,e_{j_{k}%
}\right)  \right\vert ^{q_{k}}\right)  ^{\frac{q_{k-1}}{q_{k}}}\cdots\right)
^{\frac{q_{1}}{q_{2}}}\right)  ^{\frac{1}{q_{1}}}\leq C_{k,\mathbf{r},\mathbf{q}}^{\mathbb{K}}\left\Vert A\right\Vert,
\]
for all continuous $k$-linear forms $A:X_{r_{1}}\times\cdots\times X_{r_{k}%
}\rightarrow\mathbb{K}$ whenever \eqref{clas} is fulfilled by all bounded
$m$--linear forms $T:\left(  \times_{1\leq i\leq n_{1}}X_{p_{i}^{(1)}}\right)
\times\dots\times\left(  \times_{1\leq i\leq n_{k}}X_{p_{i}^{(k)}}\right)
\rightarrow\mathbb{K}$. Let $A:X_{r_{1}}\times\cdots\times X_{r_{k}%
}\rightarrow\mathbb{K}$ be a bounded $k$-linear form. For each $i=1,\ldots,k$
the diagonal space $\overline{D}_{r_{i}}$ is complemented in $X_{p_{1}^{(i)}%
}\widehat{\otimes}^{\pi}\cdots\widehat{\otimes}^{\pi}X_{p_{n_{i}}^{(i)}}$ (see
\cite{ArFa}), and consider the diagonal projection $d_{r_{i}}$ from
$X_{p_{1}^{(i)}}\widehat{\otimes}^{\pi}\cdots\widehat{\otimes}^{\pi}X_{p_{n_{i}}%
^{(i)}}$ onto $\overline{D}_{r_{i}}$, such that $d_{r_{i}}(\sum_{j_{1}%
,\ldots,j_{n_{i}}}a_{(j_{1},\ldots,j_{n_{i}})}e_{j_{1}}\otimes\cdots\otimes
e_{j_{n_{i}}})$ is equal to $\sum_{j_{1},\ldots,j_{n_{i}}}a_{(j_{1}%
,\ldots,j_{n_{i}})}e_{j_{1}}\otimes\cdots\otimes e_{j_{n_{i}}}$ if
$j_{1}=\cdots=j_{n_{i}}$ and to $0$ otherwise. Define the $m$-linear map
$T_{A}:X_{p_{1}}\times\cdots\times X_{p_{m}}\rightarrow\mathbb{K}$ by
\begin{align*}
&T_{A}(x_{1}^{(1)},\ldots,x_{n_{1}}^{(1)},\ldots,x_{1}^{(k)},\ldots,x_{n_{k}%
}^{(k)}) \\
&:=A(u_{r_{1}}^{-1}\circ d_{r_{1}}(x_{1}^{(1)}\otimes\ldots\otimes
x_{n_{1}}^{(1)}),\ldots,u_{r_{k}}^{-1}\circ d_{r_{k}}(x_{1}^{(k)}\otimes
\ldots\otimes x_{n_{k}}^{(k)})).
\end{align*}
The following equalities give the result:
\begin{align*}
T_{A}(e_{i_{1}}^{n_{1}},\ldots,e_{i_{k}}^{n_{k}})  &  =A(u_{r_{1}}^{-1}\circ
d_{r_{1}}(\otimes_{n_{1}}e_{i_{1}}),\ldots,u_{r_{k}}^{-1}\circ d_{r_{k}%
}(\otimes_{n_{k}}e_{i_{k}}))\\
&  =A(u_{r_{1}}^{-1}(\otimes_{n_{1}}e_{i_{1}}),\ldots,u_{r_{k}}^{-1}%
(\otimes_{n_{k}}e_{i_{k}}))=A(e_{i_{1}},\ldots,e_{i_{k}}).
\end{align*}

(2) The argument is similar to the one of the case $0\leq\left\vert \frac
{1}{\mathbf{p}}\right\vert \leq\frac{1}{2}$, we just need to use
(\ref{i99jagsytsb}) instead of (\ref{ttr}).
\end{proof}

An immediate and illustrative corollary is the case $p_{1}%
=\cdots=p_{m}=p$ which can be stated in a cleaner form (see Theorem
\ref{thmain2}).

\begin{remark}
\label{555} Looking at the proof of Theorem \ref{hl_var} and choosing $k=1$
and $n_{1}=m$ we not only recover Zalduendo's and Aron-Globevnik's theorems
but we also provide an alternative proof for them. In fact, for the sake of
simplicity let us choose $p_{1}=\cdots=p_{m}=p$; let $T:X_{p}\times
\cdots\times X_{p}\rightarrow\mathbb{K}$ be a continuous $m$-linear form and
$p>m$. Denoting by $T_{L}$ the linearization of $T$ and, as usual, letting
$\frac{p}{p-m}=1$ when $p=\infty$, we have
\begin{align*}
\left(  \sum\limits_{j=1}^{\infty}\left\vert T\left(  e_{j},...,e_{j}\right)\right\vert ^{\frac{p}{p-m}}\right)  ^{\frac{p-m}{p}} & =\left(  \sum\limits_{j=1}^{\infty}\left\vert T_{L}\left(  \otimes_{m}^{\pi}e_{j}\right)\right\vert ^{\frac{p}{p-m}}\right)  ^{\frac{p-m}{p}} \\
& \leq\left\Vert T_{L}\right\Vert \left\Vert \left(  \otimes_{m}^{\pi}e_{j}\right)_{j=1}^{\infty}\right\Vert _{w,\frac{p}{p-m}}.
\end{align*}
But, from Lemma \ref{norma1p} we know that $\left\Vert \left(  \widehat{\otimes
}_{m}^{\pi}e_{j}\right)  _{j=1}^{\infty}\right\Vert _{w,\frac{p}{p-m}}=1$ and
since $\left\Vert T_{L}\right\Vert =\left\Vert T\right\Vert $ the proof is
done. Concerning the optimality of the exponents, it can be easily proved
using an idea borrowed from \cite{dimant}. In fact, consider $T_{n}%
:X_{p}\times\cdots\times X_{p}\rightarrow\mathbb{K}$ given by%
\[
T_{n}\left(  x^{(1)},...,x^{(m)}\right)  =\sum\limits_{j=1}^{n}x_{j}%
^{(1)}...x_{j}^{(m)}.
\]
Then, since $\left\Vert T_{n}\right\Vert =n^{1-\frac{m}{p}}$ and
\[
\left(  \sum\limits_{j=1}^{n}\left\vert T_{n}\left(  e_{j},...,e_{j}\right)
\right\vert ^{r}\right)  ^{\frac{1}{r}}=n^{\frac{1}{r}},
\]
we conclude that%
\[
r\geq\frac{p}{p-m}.
\]
\end{remark}

\begin{remark}
Using the canonical isometric isomorphisms for the spaces of weakly summable
sequences ($\mathcal{L}(\ell_{p};E)=\ell_{p^{*}}^{w}(E)$, $1<p<\infty$, and
$\mathcal{L}(c_{0};E)=\ell_{1}^{w}(E)$), all the aforementioned inequalities
can be translated to the theory of absolutely summing operators, motivating a
general approach the encompasses the notions of absolutely summing and
multiple summing operators.
\end{remark}

\section{Applications: Constants associated to special choices of $\Lambda$}\label{123}

For real scalars, from \cite{diniz} we know that in (\ref{u88}) we have%
\[
C_{m,\infty}^{\mathbb{R}}\geq2^{1-\frac{1}{m}},
\]
so the Bohnenblust--Hille inequality for real scalars is obviously non-contractive. In this section, as a consequence of the main result of this paper, we show that the Bohnenblust--Hille inequality is, however, somewhat \textquotedblleft almost\textquotedblright\ contractive. More precisely, we consider sums in certain sets $\Lambda$, i.e.,
\[
\left(  \sum_{i_{1},\dots,i_{k}=1}^{\infty}\left\vert T\left(  e_{i_{1}%
}^{n_{1}},\dots,e_{i_{k}}^{n_{k}}\right)  \right\vert ^{\frac{2m}{m+1}%
}\right)  ^{\frac{m+1}{2m}}\leq M_{k,m,\infty}^\mathbb{K}\left\Vert T\right\Vert ,
\]
and show that if the set $\Lambda$ is composed by a certain number of \textquotedblleft blocks\textquotedblright \ $k:=k(m)$ such that
\[
\lim_{m\to\infty}\frac{k\log k}{m}=0,
\]
then
\[
\lim_{m\to\infty}M_{k,m,\infty}^\mathbb{K}=1.
\]

A similar job is done for the Hardy--Littlewood inequalities.

\subsection{Sets $\Lambda$ for contractivity of the Bohnenblust--Hille
inequality}

It is well known that (for both real and complex scalars)%
\begin{equation}
\left(  \sum_{i_{1},\dots,i_{m}=1}^{\infty}|T(e_{i_{1}},\dots,e_{i_{m}}%
)|^{2}\right)  ^{\frac{1}{2}}\leq\Vert T\Vert\label{dos}%
\end{equation}
for all continuous $m$-linear forms $T:c_{0}\times\cdots\times c_{0}%
\rightarrow\mathbb{K}$. In fact, for every positive integer $n$, by the
Khinchin inequality for multiple sums \cite[page 701]{popa1} (since the
constant of the Khinchin inequality in this case is $1$) we have
\begin{align*}
&\left(  \sum_{i_{1},\dots,i_{m}=1}^{n}|T(e_{i_{1}},\dots,e_{i_{m}})|^{2}\right)^{\frac{1}{2}} \\
&  \leq\left(
\int\limits_{0}^{1}
\cdots
\int\limits_{0}^{1}
\left\vert \sum_{i_{1},\dots,i_{m}=1}^{n}r_{i_{1}}(t_{1})\cdots r_{i_{m}}%
(t_{m})T\left(  e_{i_{1}},\ldots,e_{i_{m}}\right)  \right\vert ^{2}%
dt_{1}\cdots dt_{m}\right)  ^{1/2} \\
&  = \left(
\int\limits_{0}^{1}
\cdots\int\limits_{0}^{1}
\left\vert T\left(  \sum_{i_{1}=1}^{n}r_{i_{1}}(t_{1})e_{i_{1}},\ldots
,\sum_{i_{m}=1}^{n}r_{i_{m}}(t_{m})e_{i_{m}}\right)  \right\vert ^{2} dt_{1}\cdots dt_{m}\right)  ^{1/2} \\
&  \leq \Vert T\Vert.
\end{align*}
The next theorem can be understood as a refinement of \eqref{u88} and shows
when inequalities of the type Bohnenblust-Hille have contractive constants as
the number of variables $m$ increases. It is worth mentioning that if $m$
increases, the number of \textquotedblleft blocks\textquotedblright\ $k$ can
be maintained constant or increased as a function of $m$. By $k=k(m)$ we mean
that $k$ can vary as a function of $m$. This trivially includes the case when
$k$ is kept constant.

\begin{theorem} \label{BH_contrac}
Let $m,k$ be positive integers with $k\leq m$ and let $n_{1},\ldots,n_{k}\in\{0,1,\ldots,m\}$ with $n_{1}+\cdots+n_{k}=m.$ Then
\[
\left(  \sum_{i_{1},\dots,i_{k}=1}^{\infty}\left\vert T\left(  e_{i_{1}%
}^{n_{1}},\dots,e_{i_{k}}^{n_{k}}\right)  \right\vert ^{\frac{2m}{m+1}%
}\right)  ^{\frac{m+1}{2m}}\leq(C_{k,\infty}^{\mathbb{K}})^{\frac{k}{m}}\left\Vert
T\right\Vert 
\]
for all continuous $m$-linear forms $T:c_{0}\times\cdots\times c_{0} \to \mathbb{K}$. Besides, if $k=k(m)$ is so that 
$$
\lim_{m\rightarrow\infty}\frac{k\log k}{m}=0,
$$
then
\[
\lim_{m\rightarrow\infty}(C_{k,\infty}^{\mathbb{K}})^{\frac{k}{m}}=1.
\]
\end{theorem}

\begin{proof}
From Theorem \ref{hl_var} we know that
\begin{equation}
\left(  \sum_{i_{1},\dots,i_{k}=1}^{\infty}\left\vert T\left(  e_{i_{1}}^{n_{1}},\dots,e_{i_{k}}^{n_{k}}\right)  \right\vert ^{\frac{2k}{k+1}}\right)  ^{\frac{k+1}{2k}}\leq C_{k,\infty}^{\mathbb{K}}\left\Vert T\right\Vert\label{non}
\end{equation}
for all continuous $m$--linear forms $T:c_{0}\times\dots\times c_{0}%
\rightarrow\mathbb{K}$. Since%
\[
\frac{1}{\frac{2m}{m+1}}=\frac{\theta}{\frac{2k}{k+1}}+\frac{1-\theta}{2}%
\]
with
\[
\theta=\frac{k}{m},
\]
by (a corollary of) the H\"{o}lder inequality, and using (\ref{dos}) and
(\ref{non}) we have%
\begin{align*}
&  \left(  \sum_{i_{1},\dots,i_{k}=1}^{\infty}\left\vert T(e_{i_{1}}^{n_{1}%
},\dots,e_{i_{k}}^{n_{k}})\right\vert ^{\frac{2m}{m+1}}\right)  ^{\frac
{m+1}{2m}}\\
&  \leq\left[  \left(  \sum_{i_{1},\dots,i_{k}=1}^{\infty}\left\vert
T(e_{i_{1}}^{n_{1}},\dots,e_{i_{k}}^{n_{k}})\right\vert ^{\frac{2k}{k+1}%
}\right)  ^{\frac{k+1}{2k}}\right]  ^{\frac{k}{m}}\left[  \left(  \sum
_{i_{1},\dots,i_{k}=1}^{\infty}\left\vert T(e_{i_{1}}^{n_{1}},\dots,e_{i_{k}%
}^{n_{k}})\right\vert ^{2}\right)  ^{\frac{1}{2}}\right]  ^{1-\frac{k}{m}}\\
&  \leq\left[  \left(  \sum_{i_{1},\dots,i_{k}=1}|T(e_{i_{1}}^{n_{1}}%
,\dots,e_{i_{k}}^{n_{k}})|^{\frac{2k}{k+1}}\right)  ^{\frac{k+1}{2k}}\right]
^{\frac{k}{m}}\Vert T\Vert\\
&  \leq(C_{k,\infty}^{\mathbb{K}})^{\frac{k}{m}}\Vert T\Vert
\end{align*}
and the inequality is proved.

Besides, using the best known estimates for
$C_{k,\infty}^{\mathbb{K}}$ (see \cite[Corollary 3.2]{bohr}) we have
\[
(C_{k,\infty}^{\mathbb{K}})^{\frac{k}{m}}\leq\left(  \alpha k^{\beta}\right)
^{\frac{k}{m}}%
\]
for suitable $\alpha,\beta>0$. Note that%
\[
\lim_{m\rightarrow\infty}\left(  \alpha k^{\beta}\right)  ^{\frac{k}{m}}=1
\]
if, and only if,
\[
\lim_{m\rightarrow\infty}\log  \left(  \alpha k^{\beta}\right)
^{\frac{k}{m}}=0,
\]
if, and only if,
\[
\lim_{m\rightarrow\infty}\frac{k}{m}\left(  \log\alpha
+\beta\log k\right)  =0.
\]
This last equality is valid because
\[
\lim_{m\rightarrow\infty}\frac{k\log k}{m}=0
\]
implies
\[ \lim_{m\rightarrow\infty}\frac{k}{m}=0.
\]
\end{proof}

\begin{example}
It is interesting to verify that
\[
k=\left\lfloor \frac{m}{\left(  \log m\right)  ^{1+\frac{1}{\log\log\log m}}%
}\right\rfloor \text{ and }k=\left\lfloor m^{1-\frac{1}{\log\log m}%
}\right\rfloor
\]
satisfy our hypotheses. This is interesting since it is written as
$k=\left\lfloor m^{1-\varepsilon_{m}}\right\rfloor $ with $\lim_{m\rightarrow
\infty}\varepsilon_{m}=0.$
\end{example}

\subsection{Sets $\Lambda$ for contractivity of the Hardy--Littlewood
inequality}

The Hardy--Littlewood inequalities for $m$--linear forms (see
\cite{hardy,pra}) are in some sense\ natural extensions of the
Bohnenblust--Hille inequality when we replace $c_{0}$ by $\ell_{p}$. These
inequalities assert that for any integer $m\geq2$ and $2m\leq p\leq\infty$,
there exists a constant $C_{m,p}^{\mathbb{K}}\geq1$ such that,
\begin{equation}
\left(  \sum_{j_{1},\ldots,j_{m}=1}^{\infty}\left\vert T(e_{j_{1}}%
,\ldots,e_{j_{m}})\right\vert ^{\frac{2mp}{mp+p-2m}}\right)  ^{\frac
{mp+p-2m}{2mp}}\leq C_{m,p}^{\mathbb{K}}\left\Vert T\right\Vert , \label{i99222}
\end{equation}
for all continuous $m$--linear forms $T\colon\ell_{p}\times\cdots\times
\ell_{p}\rightarrow\mathbb{K}$. The exponent $\frac{2mp}{mp+p-2m}$ is optimal.
Note that taking $p=\infty$ in \eqref{i99222} we recover the Bohnenblust--Hille inequality.

The constants of the Hardy--Littlewood inequality were investigated in recent
papers (see \cite{ap} and the references therein). In this section we
investigate the inequality \eqref{i99222} allowing summability by blocks, in the
lines of what was done in the previous subsection with the Bohnenblust--Hille
inequality. However, the appearance of the new parameter $p$ requires a
refinement of the techniques previously used. From now on let us simplify the notation by defining
\[
\theta := \frac{2m^{2}-4m+p}{2km-2k-2m+p}
\quad \text{ and } \quad \phi:= \frac{k}{m} \cdot \theta.
\]

\begin{theorem}\label{HL_contrac}
Let $m,k$ be positive integers with $k\leq m$ and let $n_{1},\ldots,n_{k}\in\{0,1,\ldots,m\}$ with $n_{1}+\cdots+n_{k}=m$. For $p>2m$, we have
\[
\left(  \sum_{i_{1},\dots,i_{k}=1}^{\infty}\left\vert T\left(  e_{i_{1}}^{n_{1}},\dots,e_{i_{k}}^{n_{k}}\right)  \right\vert ^{\frac{2mp}{mp+p-2m}}\right)  ^{\frac{mp+p-2m}{2mp}}\leq\left(  C_{k,\left(\frac{p}{n_1},\dots,\frac{p}{n_k}\right)}^{\mathbb{K}}\right)
^{\phi}\left\Vert
T\right\Vert,
\]
for all continuous $m$-linear forms $T:\ell_{p}\times\cdots\times\ell_{p} \to\mathbb{K}$. Moreover, if $p=p(m) \geq m^2$ and $k=k(m)$ is such that
\[
\lim_{m\rightarrow\infty}\frac{k\log k}{m}=0,
\]
then
\[
\lim_{m \to \infty} \left(C_{k,\left(\frac{p}{n_1},\dots,\frac{p}{n_k}\right)}^{\mathbb{K}}\right)^{\phi}
=1.
\]
\end{theorem}

\begin{proof}
Theorem \ref{hl_var} asserts that if $1\leq k\leq m$ and $n_{1}%
,\dots,n_{k}\geq1$ are positive integers such that $n_{1}+\cdots+n_{k}=m$,
then there is a constant $M_{k,m,p}^{\mathbb{K}}\geq1$ such that
\begin{equation}
\left(  \sum_{i_{1},\dots,i_{k}=1}^{\infty}\left\vert T\left(  e_{i_{1}}^{n_{1}},\dots,e_{i_{k}}^{n_{k}}\right)  \right\vert ^{\frac{2kp}{kp+p-2m}}\right)  ^{\frac{kp+p-2m}{2kp}}\leq M_{k,m,p}^{\mathbb{K}}\left\Vert
T\right\Vert \label{clas222}
\end{equation}
for all continuous $m$--linear forms $T:\ell_{p}\times\cdots\times\ell
_{p}\rightarrow\mathbb{K}$, and the exponent $\frac{2kp}{kp+p-2m}$ is optimal. In Theorem \ref{hl_var} it is also proved that
\begin{equation}
M_{k,m,p}^{\mathbb{K}}\leq C_{k,\left(\frac{p}{n_1},\dots,\frac{p}{n_k}\right)}^{\mathbb{K}} \label{c22}
\end{equation}
for all $1\leq k\leq m$, where $C_{k,\left(\frac{p}{n_1},\dots,\frac{p}{n_k}\right)}^{\mathbb{K}}$ is the optimal constant
of the Hardy--Littlewood inequality for $k$-linear forms on $\ell_{\frac{p}{n_1}}\times\cdots\times\ell_{\frac{p}{n_k}}$.
%From now on $C_{k}^{\mathbb{K}}$ will be just denoted by $C_{k}.$

It is obvious that
\[
\left(\sum_{i_{1},\dots,i_{k}=1}^{\infty}\left\vert T\left(e_{i_{1}}^{n_{1}},\dots,e_{i_{k}}^{n_{k}}\right)  \right\vert^{\frac{2p}{p-2m+2}}\right)^{\frac{p-2m+2}{2p}} \leq\left(  \sum_{j_{1},\ldots,j_{m}=1}^{\infty}\left\vert T(e_{j_{1}},\ldots,e_{j_{m}})\right\vert ^{\frac{2p}{p-2m+2}}\right)  ^{\frac{p-2m+2}{2p}}
\]
By \cite[Lemma 5.1]{pt} we know that
\begin{equation}
\left(  \sum_{j_{1},\ldots,j_{m}=1}^{\infty}\left\vert T(e_{j_{1}}%
,\ldots,e_{j_{m}})\right\vert ^{\frac{2p}{p-2m+2}}\right)  ^{\frac{p-2m+2}%
{2p}}\leq\left\Vert T\right\Vert \label{hhh}%
\end{equation}
for $p>2m.$ Thus, since
\[
\frac{1}{\frac{2mp}{mp+p-2m}}
= \frac{\phi}
{\frac{2kp}{kp+p-2m}}
+
\frac{1 - \phi}
{\frac{2p}{p-2m+2}},
\]
by (a corollary of) the H\"{o}lder inequality, and using \eqref{clas222}, \eqref{c22} and \eqref{hhh} we have
\[
\left(  \sum_{i_{1},\dots,i_{k}=1}^{\infty}\left\vert T\left(  e_{i_{1}}^{n_{1}},\dots,e_{i_{k}}^{n_{k}}\right)  \right\vert ^{\frac{2mp}{mp+p-2m}}\right)^{\frac{mp+p-2m}{2mp}}
\leq\left(  C_{k,\left(\frac{p}{n_1},\dots,\frac{p}{n_k}\right)}^{\mathbb{K}}\right)^{\phi}
\left\Vert T\right\Vert.
\]

Moreover, from \cite{apjfa} and \cite{bohr}, there are constants $\alpha,\beta>0$ such that
\[
C_{k,\left(\frac{p}{n_1},\dots,\frac{p}{n_k}\right)}^{\mathbb{K}}\leq \sigma_\mathbb{K}^{\frac{2(k-1)m}{p}}\left(\alpha k^{\beta}\right)^{\frac{p-2m}{p}}
\]
for all $m$, where $\sigma_\mathbb{R}=\sqrt{2}$ and $\sigma_\mathbb{C}=\frac{2}{\sqrt{\pi}}$.
Let us see that
\[
\lim_{m\rightarrow\infty}\left(  \sigma_\mathbb{K}^{\frac{2(k-1)m}{p}}\left(\alpha k^{\beta}\right)^{\frac{p-2m}{p}}\right)  ^{\phi}=1.
\]
Indeed, observe that
\[
\lim_{m\rightarrow \infty }\left( \left( \sigma_\mathbb{K}\right) ^{\frac{2(k-1)m}{p}%
}\left( \alpha k^{\beta}\right) ^{\frac{p-2m}{p}}\right) ^{\phi}=1
\]
if, and only if,
\[
\lim_{m\rightarrow \infty }\log \left( \left( \sigma_\mathbb{K} \right) ^{\frac{2(k-1)m}{p}}\left(\alpha k^{\beta}\right) ^{\frac{p-2m}{p}}\right) ^{\phi}=0
\]
if, and only if,
\begin{equation}
\lim_{m\rightarrow \infty }\left( \frac{2(k-1)k}{p} \cdot \theta \cdot
%\frac{2m^{2}-4m+p}{2km-2k-2m+p}
\log \left( \sigma_\mathbb{K}\right)
+\frac{k(p-2m)}{mp}\cdot \theta \cdot
%\frac{2m^{2}-4m+p}{2km-2k-2m+p}
\log \left(\alpha k^{\beta}\right) \right) =0. \label{33}
\end{equation}

Since
\[
\lim_{m\rightarrow\infty}\frac{k\log k}{m}=0
\]
implies
\[
\lim_{m\rightarrow\infty}\frac{k}{m}=0
\]
and since $p\geq m^2$, we have
\[
\lim_{m\rightarrow \infty }\frac{2(k-1)k}{p}=0.
\]
Moreover, observe that
\begin{equation}\label{llllls}
\sup_{m} \theta
<\infty.
\end{equation}
Thus
\[
\lim_{m\rightarrow \infty }\frac{2(k-1)k}{p} 
\cdot \theta \cdot
%\frac{2m^{2}-4m+p}{2km-2k-2m+p}
\log \left(\sigma_\mathbb{K}\right)=0
\]
and \eqref{33} happens if, and only if,
\[
\lim_{m\rightarrow \infty }
\frac{k(p-2m)}{mp} \cdot \theta \cdot
%\frac{2m^{2}-4m+p}{2km-2k-2m+p}
\log \left( \alpha k^{\beta}\right) =0.
\]
Observe that
\begin{align*}
\lim_{m\rightarrow \infty }&
\left[\frac{k(p-2m)}{mp}
\cdot \theta \cdot
%\frac{2m^{2}-4m+p}{2km-2k-2m+p}
\log \left( \alpha k^{\beta}\right)
\right] \\
& =\lim_{m\to\infty} \left[ \frac{p-2m}{p}
\cdot \theta \cdot
%\frac{2m^{2}-4m+p}{2km-2k-2m+p}
\frac{k\log \left( \alpha k^{\beta}\right)}{m}
\right] \\
& =\lim_{m\to\infty} \left[ \frac{p-2m}{p}
\cdot \theta \cdot
%\frac{2m^{2}-4m+p}{2km-2k-2m+p}
\left(\frac{k\log\alpha}{m} + \frac{\beta k \log k}{m} \right)
\right]
\end{align*}
Using \eqref{llllls} again and the boundedness of $(p-2m)/p$ we conclude that
\[
\lim_{m\to\infty} \left[ \frac{p-2m}{p}
\cdot \theta \cdot
%\frac{2m^{2}-4m+p}{2km-2k-2m+p}
\left(\frac{k\log\alpha}{m} + \frac{\beta k \log k}{m} \right)
\right]
=0,
\]
and the proof is done.
\end{proof}

%Now we have a partial converse of Theorem \ref{ineq+lim}:

\end{document}